\documentclass[11pt]{amsart}
\usepackage{amsmath,amsfonts,amssymb,amsthm,epsfig,color, mathrsfs} 
\usepackage{hyperref} 
\usepackage{bbm}

\newcommand{\bb}{\mathbb}

\newcommand{\CC}{\bb C}
\newcommand{\h}{\bb H}
\newcommand{\Z}{\bb Z}
\newcommand{\ZZ}{\bb Z}
\newcommand{\R}{\bb R}
\newcommand{\RR}{\bb R}

\newcommand{\NN}{\bb N}

\newcommand{\Ai}{\operatorname{Ai}}

\newcommand{\D}{\mathscr{D}}
\newcommand{\Ce}{\mathscr{C}}

\newcommand{\eL}{\mathcal{L}}

\newcommand{\Ga}{\Gamma}

\newcommand{\wrt}[1]{\mathrm{d}{#1}}

\newcommand{\PSL}{\operatorname{PSL}}

\newcommand{\im}{\operatorname{Im}}

\newtheorem{Theorem}{Theorem}
\numberwithin{Theorem}{section}

\newtheorem{coro}[Theorem]{Corollary}

\newtheorem{theo}[Theorem]{Theorem}
\newtheorem{prop}[Theorem]{Proposition}

\newtheorem{lemm}[Theorem]{Lemma}

\renewcommand{\Im}{\operatorname{Im}}
\renewcommand{\Re}{\operatorname{Re}}

\newtheorem*{lemma*}{Lemma}
\newtheorem*{question*}{Question}

\newtheorem*{theorem*}{Theorem}

\theoremstyle{remark}

\newtheorem{rema}[Theorem]{\sc Remark}
\newtheorem*{rema*}{\sc Remark}
\numberwithin{equation}{section}
\begin{document}

\title[Eisenstein series and the $K$-Bessel function]{Eisenstein series and an asymptotic for the $K$-Bessel function}
\author{Jimmy Tseng}
\thanks{The author was supported by EPSRC grant EP/T005130/1.}
\address{Department of Mathematics, University of Exeter, Exeter, EX4 4QF, UK}
\email{j.tseng@exeter.ac.uk}
\begin{abstract}  We produce an estimate for the $K$-Bessel function $K_{r  + i t}(y)$ with positive, real argument $y$ and of large complex order $r+it$ where $r$ is bounded and $t = y \sin \theta$ for a fixed parameter $0\leq \theta\leq \pi/2$ or $t= y \cosh \mu$ for a fixed parameter $\mu>0$.  In particular, we compute the dominant term of the asymptotic expansion of $K_{r  + i t}(y)$ as $y \rightarrow \infty$.  When $t$ and $y$ are close (or equal), we also give a uniform estimate.

As an application of these estimates, we give bounds on the weight-zero (real-analytic) Eisenstein series $E_0^{(j)}(z, r+it)$ for each inequivalent cusp $\kappa_j$ when $1/2 \leq r \leq 3/2$.



%
%

\end{abstract}

\maketitle
\tableofcontents

\section {Introduction}\label{secBackground}



The $K$-Bessel function, $K_{r  + i t}(y)$, (see~\ref{eqnKInTermsOfI} for the definition) appears in a number of ways in mathematics such as in the Fourier expansion of Eisenstein series (for background on Eisenstein series, see~\cite{Ku} for example), and these series are important automorphic functions (namely, functions invariant under a cofinite Fuchsian group) because they are eigenfunctions of the non-Euclidean Laplacian (i.e. the operator $D:=y^2\left(\frac{\partial^2}{\partial x^2}+ \frac{\partial^2}{\partial y^2} \right)$).\footnote{Other names for the $K$-Bessel function also exist in the literature such as, for example, the Macdonald-Bessel function, the modified Bessel function of the second kind and the modified Bessel function of the third kind.}  In this paper, we will produce bounds for $K_{r  + i t}(y)$ in the novel case of positive, real argument $y$ and of large complex order $r+it$ where $r$ is bounded and $t$ varies linearly with $y$ in all possible ways.  In particular, we compute the dominant term of the asymptotic expansion of $K_{r  + i t}(y)$ as $y \rightarrow \infty$ for the two cases $t = y \sin \theta$ for a fixed parameter $0\leq \theta\leq \pi/2$ (Theorem~\ref{thmFirstCaseMonoKBessel}) or $t= y \cosh \mu$ for a fixed parameter $\mu>0$ (Theorem~\ref{thmSecondCaseOscillKBessel}).  (The case $t<0$ is also handled as Remark~\ref{rmkNegativeTIsOK} shows.)  Note that, thus, our result is for $y$ and $t$ both approaching infinity.  Except for the case of $\theta=\pi/2$, we prove Theorems~\ref{thmFirstCaseMonoKBessel} and~\ref{thmSecondCaseOscillKBessel} using Laplace's method (see~\cite[Page~127, Theorem~7.1]{Olv} or~\cite{Nem} for example) in Section~\ref{subsecBndsBessel}.

Theorems~\ref{thmFirstCaseMonoKBessel} and~\ref{thmSecondCaseOscillKBessel} are nonuniform results.  In the case where $t$ and $y$ are nearly equal, we will find that there are two relevant saddle points.  When $t/y$ approach $1$, the two saddle points coalesce and these results go to infinity.  Consequently, a uniform result in this case is highly desirable.  We give such a uniform result (Theorems~\ref{thmFirstCaseMonoKBesselUnif} and~\ref{thmSecondCaseOscillKBesselUnif}) in Section~\ref{subsecBndsBesselUnif}.  The uniform result is important not only for the completeness of the estimates for the $K$-Bessel function but also for applications.


One such application for estimates on $K$-Bessel functions is the study of Eisenstein series.  As an application of our results, we will, in Section~\ref{secBndEisenstein}, give bounds on the weight-zero Eisenstein series $E_0^{(j)}(z, r+it)$ for each inequivalent cusp $\kappa_j$ when $1/2 < r \leq 3/2$.  (For the case $r=1/2$, we will use known estimates on the $K$-Bessel function, $K_{i t}(y)$, to give bounds on the Eisenstein series, $E_0^{(j)}(z, 1/2+it)$.)  Already, our nonuniform results suffice to give bounds on the Fourier coefficients of these Eisenstein series (Theorem~\ref{lemmBndSumcnsquares}) when $1/2 < r \leq 3/2$.  However, to bound the Eisenstein series themselves (Theorem~\ref{propEisenBiggerHalf}), it is necessary, when $1/2 < r \leq 3/2$, to use our uniform results.

\subsection{Statement of results}

Let $\nu := r +it$.  Our first two results (Theorems~\ref{thmFirstCaseMonoKBessel} and~\ref{thmSecondCaseOscillKBessel}) together give an asymptotic for $K_{\nu}(y)$ for large order but bounded $r$ (so $|t|$ grows to infinity) and positive, real argument $y$.  There are two cases:  $y \geq t \geq 0$ (Theorem~\ref{thmFirstCaseMonoKBessel}) and $0 < y < t$ (Theorem~\ref{thmSecondCaseOscillKBessel}).  When $t <0$, see Remark~\ref{rmkNegativeTIsOK}.

We note that more terms of the asymptotic expansions found in Theorems~\ref{thmFirstCaseMonoKBessel},~\ref{thmSecondCaseOscillKBessel},~\ref{thmFirstCaseMonoKBesselUnif},~and~\ref{thmSecondCaseOscillKBesselUnif} could be computed using the techniques in this paper; however, these computations quickly become tedious and are omitted.

\begin{theo}\label{thmFirstCaseMonoKBessel}
Let $M\geq0$ and $0\leq\theta \leq \pi/2$ be fixed real numbers.  Let $|r|\leq M$, $0< y \in \RR$, and \begin{align}
t  = y \sin \theta.\end{align}  Then \begin{align*}
K_\nu(y) = \begin{cases} \sqrt{\frac{\pi}{2 y \cos \theta}}e^{-y (\cos \theta +\theta \sin \theta )}e^{ir\theta} +O\left(y^{-3/2}e^{-y (\cos \theta +\theta \sin \theta )} \right)& \text{ if } 0 \leq \theta < \frac \pi 2 \\ e^{-\frac \pi 2 y +i\frac \pi 2 r}y^{-1/3}\frac{\Gamma(\frac 1 3)} {2^{\frac 2 3}3^{\frac 1 6}} +O\left(y^{-2/3}e^{-\frac \pi 2 y +i\frac \pi 2 r} \right)& \text{ if } \theta = \frac \pi 2\end{cases}  \end{align*} as $y \rightarrow \infty$.  Here, the implied constants depend on $\theta$ and $M$ for the case $0 \leq \theta < \frac \pi 2$ and on $M$ for the case $\theta = \frac\pi 2$.
\end{theo}

\begin{rema}
In the special case of purely imaginary order, our result agrees with standard results.  As examples, see \cite[Page 87 (18)]{EMOT}~and~\cite[(14)]{BST} for the case $0 < \theta < \frac \pi 2$ and~\cite[Pages~78,~247]{Wa}~and~\cite[(14)]{BST} for the case $\theta = \frac \pi 2$.  Also, note that the $\Gamma$ in the statement of the theorem refers to the gamma function.

\end{rema}

\begin{theo}\label{thmSecondCaseOscillKBessel}
Let $M \geq0$ and $\mu > 0$ be fixed real numbers.  Let $|r|\leq M$, $0< y \in \RR$, and \begin{align}
t  = y \cosh \mu.\end{align}  Then \begin{align*}
K_\nu(y) = \sqrt{\frac{2\pi}{y \sinh \mu}}e^{-y \frac \pi 2 \cosh \mu +i r \frac \pi 2}& \left[\cosh(r \mu) \sin\left(\frac \pi 4 - y \left(\sinh \mu  - \mu \cosh \mu \right)\right) \right. \\ &\left.\quad  -i \sinh(r \mu)\cos\left(\frac \pi 4 - y \left(\sinh \mu  - \mu \cosh \mu \right)\right)\right] \\ &+O\left(y^{-3/2}e^{-y\left( \frac \pi 2 \cosh \mu + i \left(\sinh \mu  - \mu \cosh \mu \right)\right)} \right) \\&+O\left(y^{-3/2}e^{-y\left( \frac \pi 2 \cosh \mu - i \left(\sinh \mu  - \mu \cosh \mu \right)\right)} \right)
\end{align*} as $y \rightarrow \infty$.  Here, the implied constants depend on $\mu$ and $M$.
\end{theo}
\begin{rema}\label{rmkClassicalResultCase2ForImagOrder}
Since $y \sinh \mu = \sqrt{t^2 - y^2}$ and $\mu = \cosh^{-1}\left( \frac t y\right)$ hold, our result, in the special case of purely imaginary order, reduces to the standard result for purely imaginary order (see~\cite[Page~88~(19)]{EMOT} for example), namely: \[K_{it}(y) \sim \sqrt{2\pi}(t^2 - y^2)^{-\frac 1 4}e^{-t \frac \pi 2}\sin\left(\frac \pi 4 - (t^2 - y^2)^{\frac 1 2}  + t \cosh^{-1}\left( \frac t y\right) \right),\]as $y \rightarrow \infty$.  

\end{rema}

Our next two results (Theorems~\ref{thmFirstCaseMonoKBesselUnif} and~\ref{thmSecondCaseOscillKBesselUnif}) give a {\em uniform} asymptotic for $K_{\nu}(y)$ for large order but bounded $r$ and positive, real argument $y$ in the case where $t$ and $y$ are nearly equal (or equal).   Here, there are also two cases:  $y \geq t \geq 0$ (Theorem~\ref{thmFirstCaseMonoKBesselUnif}) and $0 < y < t$ (Theorem~\ref{thmSecondCaseOscillKBesselUnif}).  When $t <0$, see Remark~\ref{rmkNegativeTIsOK}.  Note that $\Ai(\cdot)$ is the Airy function.

\begin{theo}\label{thmFirstCaseMonoKBesselUnif}
Let $M\geq0$ and $0< \theta \leq \frac \pi 2$ be real numbers.  Let $|r|\leq M$, $0< y \in \RR$, and \begin{align}
t  = y \sin \theta.\end{align}  Then there exists a (small) $\theta_0>0$, which does not depend on $t$ or $y$, such that, for all $\frac \pi 2 - \theta_0 \leq \theta \leq \frac \pi 2$, we have  \begin{align*}
K_\nu(y) &=  \ \frac {\pi \sqrt 2}{y^{1/3}} e^{-y \frac \pi 2 \sin \theta + i r \frac \pi 2} \cos\left(r \theta - r \frac \pi 2\right) \left(\frac \zeta {\cos^2 \theta} \right)^{1/4} \Ai\left( y^{2/3} \zeta \right) \\ & - \frac{i \pi \sqrt 2}{y^{2/3}} e^{-y \frac \pi 2 \sin \theta + i r \frac \pi 2} \sin\left(r \theta - r \frac \pi 2\right)\zeta^{-1/2} \left(\frac \zeta {\cos^2 \theta} \right)^{1/4} \Ai'\left( y^{2/3} \zeta \right) \\ &+ O\left(\frac{\Ai\left( y^{2/3} \zeta \right) e^{-y \frac \pi 2 \sin \theta }}{y^{4/3}}\right) + O\left(\frac{\Ai'\left( y^{2/3} \zeta \right)e^{-y \frac \pi 2 \sin \theta}}{y^{5/3}}\right) \end{align*} as $y \rightarrow \infty$.  Here $\zeta = \left[\frac 3 2  \left(\theta \sin \theta + \cos \theta - \frac \pi 2 \sin \theta \right) \right]^{2/3}$ is a nonnegative real number and the implied constants depend on $\theta_0$ and $M$.

\end{theo}

\begin{theo}\label{thmSecondCaseOscillKBesselUnif}
Let $M \geq0$ and $\mu \geq  0$ be real numbers.  Let $|r|\leq M$, $0< y \in \RR$, and \begin{align}
t  = y \cosh \mu.\end{align}  Then there exists a (small) $\mu_0>0$, which does not depend on $t$ or $y$, such that, for all $0 \leq \mu \leq \mu_0$, we have  \begin{align*}
K_\nu(y) = & \ \frac {\pi \sqrt 2}{y^{1/3}} e^{-y \frac \pi 2 \cosh \mu + i r \frac \pi 2} \cosh\left(r \mu\right) \left(\frac \zeta {-\sinh^2 \mu} \right)^{1/4} \Ai\left( y^{2/3} \zeta \right) \\ & - \frac{ \pi \sqrt 2}{y^{2/3}} e^{-y \frac \pi 2 \cosh \mu + i r \frac \pi 2} \sinh\left(r \mu\right)\zeta^{-1/2} \left(\frac \zeta {-\sinh^2 \mu} \right)^{1/4} \Ai'\left( y^{2/3} \zeta \right) \\ &+ O\left(\frac{\Ai\left( y^{2/3} \zeta \right) e^{-y \frac \pi 2 \cosh \mu }}{y^{4/3}}\right) + O\left(\frac{\Ai'\left( y^{2/3} \zeta \right)e^{-y \frac \pi 2 \cosh \mu}}{y^{5/3}}\right)\end{align*} as $y \rightarrow \infty$.  Here $\zeta = -\left[\frac 3 2  \left(\mu \cosh \mu - \sinh \mu \right) \right]^{2/3}$ is a nonpositive real number and the implied constants depend on $\mu_0$ and $M$.
\end{theo}

 \begin{rema}  We make a few observations.
 \begin{enumerate}
\item  
When $r=0$, our result agrees with the standard result by Balogh~\cite{Bal66}. To see this, let us use $\widetilde{\zeta}$ to denote $\zeta$ from~\cite{Bal66} to distinguish it from our use of $\zeta$.  For the first case, letting $\widetilde{\theta} = \theta - \pi/2$, we note that $\sec^{-1}(\sec \widetilde{\theta}) = - \widetilde{\theta}$ as $\widetilde{\theta} <0$, which yields via a short computation that $\zeta = \widetilde{\zeta} \cos^{2/3} \widetilde{\theta}$.  As $y \cos \widetilde{\theta} = t$, the agreement follows.  For the second case, we note that $\sec^{-1}(1/\cosh \mu) = \cos^{-1}(\cosh \mu) = i \mu$, which yields the nonpositive real number $\zeta = \widetilde{\zeta} \cosh^{2/3} \mu$ and agreement.

 \item The expressions are defined when $\theta \rightarrow \frac \pi 2$ and when $\mu \rightarrow 0$ by Taylor approximation.
 
 \item When $r \neq 0$, there are order $y^{-2/3}$ terms, unlike when $r=0$.

\end{enumerate}

\end{rema}

We also give a result for small $y$, which will be applied in the computation of our bounds for the Eisenstein series.

\begin{prop}\label{lemmUBndsOnK}
For $3/2 \geq r \geq 1/2$, $|t| \geq t_0$, and $0<y <1$, we have \[K_{r - 1/2 +i t}(y) = O( y^{1/2-r} e^{-|t|\pi/2} |t|^{r-1})\] where the implied constant depends only $t_0$ and is uniformly bounded for all large enough $t_0$.
\end{prop}

\begin{rema}
Here, $t_0\geq1$ is chosen to be a fixed large constant (large enough to use the first term in the Stirling asymptotic series for the gamma function for the approximation in the proof of the proposition below).

\end{rema}

 Let $z:=x+iy, s:=r+it \in \CC$.  As an application of our above results, we compute bounds on the Eisenstein series for large enough $|t|$. Let $G:= \PSL_2(\R)$, $\Ga \subset G$ be a cofinite Fuchsian group, and $\h$ be the upper-half plane model of the hyperbolic plane (i.e. with the Poincare metric).  The group $G$ acts transitively on the left of $\h$ via M\"obius transformations and, moreover, these actions are orientation-preserving isometries.  We assume that $\Gamma \backslash \h$ has at least one cusp, that one of these cusps is located at $\infty$, and that the cusp $\kappa_1:=\infty$ (called the {\em standard cusp}) has stabilizer \[\Ga_1:=\Ga_\infty := \left\{ \left( \begin{array}{cc}
1 & b \\
0 & 1 \end{array} \right) \bigg{\vert} \ b \in \Z \right\}\] in $\Ga$.  As $\Ga_\infty$ acts on the unit strip $[0,1] \times (0, \infty)$ to tesselate $\h$, the quotient group $\Ga_\infty \backslash \Ga$ tessellates the unit strip so as to agree with the tessellation of $\h$ given by $\Ga$ and we have a canonical\footnote{See~\cite[Chapter~6, Page~5]{He} for the definition of \textit{canonical}.} fundamental domain $F$ that extends to infinity for the $\Ga_\infty \backslash \Ga$ action---to determine $F$, let the real part of the points of $F$ range between $0$ and $1$, inclusive of $0$.  Often we will consider the topological closure $\overline{F}$.

There are, in general, a finite number of inequivalent cusps $\{\kappa_j\}_{j=1}^q \subset \RR \cup \{\infty\}$, and the stabilizer in $\Gamma$ of a cusp $\kappa_j$ is a parabolic subgroup $\Gamma_j$ (see, for example,~\cite[Chapter~6]{He} for the definition of inequivalent cusps).  For each inequivalent cusp, we choose $\sigma_j \in G$ such that $\sigma_j(\kappa_j) =\infty$, namely taking the cusp $\kappa_j$ into the standard cusp.  (We always choose $\sigma_1$ to be the identity.)  Note that $\sigma_j$ is not in $\Ga$ for any $j \in \{2, \cdots, q\}$.  By modifying $\sigma_j$ for $j \in \{2, \cdots, q\}$, we can ensure that \begin{align}\label{eqn:FullCuspInY}
 \sigma_j(\overline{F}) \cap \{z \in \h : y \geq B\} = [0,1] \times [B, \infty) \end{align} holds for all $j \in \{1, \cdots, q\}$ and for all $B \geq B_0>1$ (see~\cite[(2.2)]{St} or~\cite[Page 268]{He}).  Here $B_0$ is a fixed constant depending only on $\Ga$.  Let us denote the $j$-th cuspidal region in $\overline{F}$ by $\Ce_{j,B}$:  \[\Ce_{j,B} := \sigma_j^{-1}\left( [0,1] \times [B, \infty)\right) \subset \overline{F}.\]  And define the bounded region of $\overline{F}$ by \[F_B:=\overline{F} - \bigcup_{j=1}^q \Ce_{j,B}.\]


There is an Eisenstein series $E^{(j)}(z, s)$ of weight $0$ for each inequivalent cusp~\cite[Definition~3.5, page~280]{He}:  \[E^{(j)}(z,s) :=E_0^{(j)}(z, s):=\sum_{\sigma \in \Ga_j \backslash \Ga}  (\im(\sigma_j  \sigma z))^s \quad \quad E(z,s):= E^{(1)}(z,s):= E_0^{(1)}(z,s).\]  The Fourier expansion at the standard cusp is the following (see ~\cite[Lemma~2.6]{MS3} or~\cite[Page~280]{He} for example):  \begin{align}\label{eqnFourExpInYEisen2}
E^{(j)}_0(z,r +i t) = \delta_{j1}y^{r+it} &+\varphi_{j1}(r +it) y^{1-r - it}   \\\nonumber & + \sum_{n\neq0} \psi_{n,j} (r+it) \sqrt{y} K_{r-1/2+it}(2 \pi|n|y)e^{2 \pi i n x} \end{align} where $\varphi_{j1}(r +it)$ is an element in the scattering matrix $\Phi(r+it) = (\varphi_{jk}(r+it))$ (cf.~\cite[Chapter 8]{He}) and $\psi_{n,j} (r+it)$ are the Fourier coefficients.  Since $E^{(j)}_0(z,r+it)$ has no poles for $|t|\geq1$ (see~\cite{Ku} and~\cite{St}), let $c_n := \psi_{n,j} (r+it)$.  


We first give a bound on the Fourier coefficients of the Eisenstein series, the proof of which only requires our nonuniform bounds on the $K$-Bessel function (Theorem~\ref{thmSecondCaseOscillKBessel} in particular).

\begin{theo}\label{lemmBndSumcnsquares}
Let $t_0\geq B_0$ be a large constant.  For $N \geq 1$, $3/2\geq r > 1/2$, and $|t| \geq t_0$, we have \[\sum_{1 \leq |n|\leq N}|c_n|^2 = O\left(e^{|t|\pi} (N+|t|)\right)\left\{\omega(t) + \left(|t| + \frac{N}{|t|}\right)^{2r-1}  \right\},\] where the implied constant depends only on the lattice subgroup $\Ga$ and $t_0$.
\end{theo}

\begin{rema}
Note that~\cite[Proposition~4.1]{St} gives a bound for the case of $r=\frac 1 2$.  Here, $\omega(t)$ denotes the spectral majorant function whose properties are $\omega(-R) = \omega(R) \geq 1$ and \begin{align}
\int_{-T}^{T} \omega(R)~\wrt R= O(T^2) 
\end{align} as $|T| \rightarrow \infty$~\cite[Pages~161, 299, 315]{He}.  The implied constant depends only on the lattice subgroup $\Ga$. 

\end{rema}  

Finally, we give a bound on the Eisenstein series themselves, the proof of which requires our bounds on the Fourier coefficients and on the $K$-Bessel function.  Note that our uniform bound for the $K$-Bessel function is essential here.

\begin{theo}\label{propEisenBiggerHalf}
Let $j \in \{1, \cdots, q\}$, $t_0\geq B_0$ be a large constant, $|t| \geq t_0$, $\frac 3 2 \geq r \geq \frac 1 2$, $y >0$, and $\varepsilon >0$.  Then, we have \begin{align*}
&E^{(j)}_0(z, r + it) = \\ &\begin{cases} 
\delta_{j1}y^{1/2+it} + O(y^{1/2})+O\left(y^{-1/2 -\varepsilon}\sqrt{\omega(t)} |t|^{1+\varepsilon} \right) & \textrm{ if } r = \frac 1 2 \textrm { and } 0<y<1 \\
 \delta_{j1}y^{r+it} + O(y^{1-r})+O(y^{1-r})\left(\left(\frac{|t|} y\right)^{r+1/2}  
+ \frac{|t|} y \sqrt{\omega(t)} \right) & \textrm{ if } 1 \geq r > \frac 1 2 \textrm { and } 0<y<1 \\  
\delta_{j1}y^{r+it} + O(y^{1-r})+O\left(\left(\frac{|t|} y\right)^{2r-1/2}  
+ \left(\frac{|t|} y\right)^r\sqrt{\omega(t)} \right) & \textrm{ if } \frac 3 2 \geq r > 1 \textrm { and } 0<y<1 
\\ \delta_{j1}y^{1/2+it} + O(y^{1/2}) + O\left(
 |t|^{1+\varepsilon}\sqrt{\omega(t)} \right) & \textrm{ if }  r = \frac 1 2 \textrm { and } 1\leq y \leq \frac{|t|} 2 
\\ \delta_{j1}y^{r+it} + O(y^{1-r}) + O\left(|t|^{r+1/2}  
+ |t| \sqrt{\omega(t)} \right) & \textrm{ if } 1 \geq r > \frac 1 2 \textrm { and } 1\leq y \leq \frac{|t|} 2  
\\  \delta_{j1}y^{r+it} ++ O(y^{1-r})+O\left(|t|^{2r-1/2}  
+ |t|^r \sqrt{\omega(t)} \right) & \textrm{ if } \frac 3 2 \geq r > 1 \textrm { and } 1\leq y \leq \frac{|t|} 2 
\\ \delta_{j1}y^{1/2+it} + O(y^{1/2})+O\left(e^{|t|\frac \pi 2-2 \pi y}\right) \left(|t|^{-1/2+\varepsilon}\sqrt{\omega(t)} \right) & \textrm{ if }  r = \frac 1 2 \textrm { and }  \frac{|t|} 2 < y
\\  \delta_{j1}y^{r+it} + O(y^{1-r})+O\left(e^{|t|\frac \pi 2-2 \pi y}\right) \left(\sqrt{|t|}+ \frac{\sqrt{\omega(t)}}{\sqrt{|t|}}\right) & \textrm{ if } \frac 3 2 \geq r > \frac 1 2 \textrm { and }  \frac{|t|} 2 < y\end{cases}  \end{align*} where the implied constants depend only on the lattice subgroup $\Ga$ and $t_0$.  
\end{theo}

\begin{rema}  For $ \frac{|t|} 2 < y$, we have an alternative formulation of the theorem: \begin{align*}
&E^{(j)}_0(z, r + it) = \\ &\begin{cases}  \delta_{j1}y^{1/2+it} + O(y^{1/2})+O\left(e^{|t|\frac \pi 2-2 \pi y}\right) \left(y^{-1}|t|^{1/2+\varepsilon}\sqrt{\omega(t)} \right) & \textrm{ if }  r = \frac 1 2 \textrm { and }  \frac{|t|} 2 < y \\ 
\delta_{j1}y^{r+it} + O(y^{1-r})+O\left(e^{|t|\frac \pi 2-2 \pi y}\right) \left(y^{-1} \left(|t|^{3/2}+\sqrt{|t|\omega(t)}\right)\right) & \textrm{ if } \frac 3 2 \geq r > \frac 1 2 \textrm { and }  \frac{|t|} 2 < y\end{cases}  \end{align*}  
\end{rema}


These bounds on the Eisenstein series give the following corollary:

\begin{coro}
 Let $j \in \{1, \cdots, q\}$, $t_0\geq B_0$ be a large constant, $|t| \geq t_0$ and $\frac 3 2 \geq r \geq \frac 1 2$.  Then, as $y \rightarrow \infty$, $E^{(j)}_0(z, r + it)$ decays exponentially (like $y^{-1}e^{-2 \pi y}$) to the constant term of its Fourier expansion at a cusp.
\end{coro}
\begin{proof}
 The result is immediate for the Fourier expansion at the standard cusp.  For the Fourier expansion at other cusps, the analog of Theorem~\ref{propEisenBiggerHalf} holds with analogous proof.  This gives the desired result.
\end{proof}

\begin{rema}  We now compare our bounds for the Eisenstein series with those of others.  \begin{enumerate}
\item For $r >1$ and $t \in \RR$, it can be shown that $E^{(j)}_0(z, r + it) = \delta_{j1}y^{r+it} +O(y^{1-r}) + O((1+y^{-r})e^{-2\pi y})$ where the later implied constant depends on $t$ (and the lattice $\Ga$)~\cite[Corollary~3.5]{Iwa02}.  Our bound, however, makes the $t$ dependance (for $|t| \geq t_0$) explicit.  Also, as $y \rightarrow \infty$, our result gives faster decay ($y^{-1}e^{-2 \pi y}$ versus $e^{-2 \pi y}$) to the constant term of the Fourier expansion.

\item For $r= 1/2$, there has been some recent interest on bounds for the Eisenstein series.  In particular, the sup-norm problem for certain eigenfunctions has had much interest (see~\cite{IwaSar95, BHMM20, Tem15} for example).  Specifically, for Eisenstein series, there are  recent results in~\cite{You18, HuXu17, As19} of which the most relevant for us is the result by Huang and Xu (generalizing the earlier result of Young) for the modular group $\Ga = \PSL_2(\ZZ)$~\cite[Theorem~1.1]{HuXu17}: \begin{align*}\label{eqnHuangXu}
 E_0(z, 1/2 + it) = y^{1/2+it} +O(y^{1/2}) + O(y^{-1/2} +t^{3/8+\varepsilon}). \end{align*}  (As $\Ga = \PSL_2(\ZZ)$ has only one cusp, we have dropped the superscript notation in the Eisenstein series.)  Note that the bound on the Eisenstein series given by Huang and Xu does not decay exponentially to the constant term of its Fourier expansion as $y \rightarrow \infty$.  Our bound, however, has this exponential decay.

\end{enumerate}

\end{rema}

%

\subsection{Outline of paper}  Section~\ref{subsecBndsBessel} is devoted to the proof of Theorems~\ref{thmFirstCaseMonoKBessel} and~\ref{thmSecondCaseOscillKBessel}.  Section~\ref{subsecBndsBesselUnif} is the devoted to the prove of Theorems~\ref{thmFirstCaseMonoKBesselUnif} and~\ref{thmSecondCaseOscillKBesselUnif}.  Section~\ref{secBndKBesselSmally} gives a proof of Proposition~\ref{lemmUBndsOnK}.  Finally, Section~\ref{secBndEisenstein} gives a proof Theorems~\ref{lemmBndSumcnsquares} and~\ref{propEisenBiggerHalf}.

\subsection*{Acknowledgements}  I would like to thank one of the referees for pointing me to Laplace's method and all of the referees for their comments.

\section{Bounds for  $K_\nu(y)$ where $\Im(\nu)$ large, $\Re(\nu)$ bounded, and $y$ is real and positive}\label{subsecBndsBessel}  


For background on asymptotic expansions, see~\cite{Cop} (especially Chapter~7) for example.  The saddle points and paths of steepest descent for the function $K_{it}(y)$ (i.e. purely imaginary order) have been obtained by N.~M.~Temme~\cite{Tem}.  The saddle points and paths of steepest descent for our function $K_{\nu}(y)$ are the same as we now show.  In addition, we give a proof of the dominant behavior.

In this section (Section~\ref{subsecBndsBessel}), let us set \[\nu:= r +i t\] where $r, t \in \RR$.  An integral representation for $K_{\nu}(z)$ (see~\cite[Page~182~(7)]{Wa} for example) is \begin{align}\label{IntRepKBessel}
K_{\nu}(z) = \frac 1 2 \int_{-\infty}^\infty e^{-z  \cosh R - \nu R } ~\wrt R = \frac 1 2 \int_{-\infty}^\infty e^{-z  \cosh R + \nu R } ~\wrt R
\end{align} where $z \in \CC \backslash \{0\}$ such that $|\arg(z)|< \frac{\pi}{2}$.  There are two cases:  $y \geq t \geq 0$ and $0 < y \leq t$.

\begin{rema}\label{rmkNegativeTIsOK}
Note that if $t <0$, then applying (\ref{IntRepKBessel}) allows us to be in one of these two cases.

\end{rema}

\subsection{First case:  $y \geq t \geq 0$} \label{subsecFirstCaseMonoKBessel}

\begin{proof}[Proof of Theorem~\ref{thmFirstCaseMonoKBessel}]

Let us first consider the case $0< \theta<\pi/2$.  Using (\ref{IntRepKBessel}), we have \begin{align}\label{IntRepKBessel2}
 K_\nu(y) = \frac 1 2 \int_{-\infty}^\infty e^{-y \varphi(R)} e^{rR} ~\wrt R  \end{align}where \[\varphi(R):= \cosh R - i R \sin \theta.\]
 
 The saddle points (values of $R$ for which $\varphi'(R)=0$) are as follows~\cite{Tem} (see also~\cite[Section~2.1]{BST}
): \[R_k := i\left( (-1)^k \theta + k\pi\right), \quad k \in \ZZ.\]

Let us now write $R=u+iw$ and thus we have \begin{align*}
\Re(-\varphi(R)) =& -\cosh u \cos w - w \sin \theta \\
  \Im(-\varphi(R)) =& -\sinh u \sin w + u \sin \theta \end{align*}

The path of steepest descent through the saddle point $R_0 = i \theta$ is given by $ \Im(-\varphi(R)) =  \Im(-\varphi(R_0))$ and is the following curve~\cite{Tem}: \[w = \arcsin\left(\sin \theta \frac{u}{\sinh u} \right), \quad -\infty < u < \infty.\]  We remark that $w'(0)=0$ and that $w'(u)$ is bounded over all $-\infty < u <\infty$.

We will apply Laplace's method, which can be found at~\cite[Page~127, Theorem~7.1]{Olv}.  Using (\ref{IntRepKBessel2}), the path of steepest descent as the integration path used in Laplace's method, and $R_0=i\theta$ as the saddle point, we see that assumptions (i) -- (iv) of Laplace's method is satisfied.  

It remains to show that the final condition (v) is also satisfied.  We know that the integration path is a path of steepest descent because along it $\Im(\varphi(u+iw))$ is constant and, when $u \rightarrow \pm \infty$, we have that $\Re(\varphi(u+iw)) \rightarrow \infty$.  As $R_0$ is the only saddle point lying on the path of steepest descent, then $R_0$ is a global minimum on the path (see~\cite[Page~66]{Cop}).  Thus, condition (v) is satisfied and we may apply Laplace's method to obtain the desired result.

The case $\theta =0$ is a simplification of the case $0< \theta< \pi/2$.

Let us now consider the case $\theta = \pi/2$ (or, equivalently, $t=y$).   Apply~\cite[Page~78 (8) and Page~247 (5)] {Wa} to obtain \[K_{r+iy}(y) \sim \frac 1 2 \pi i e^{\frac 1 2(r+iy)\pi i} \left(- \frac 2{3\pi} e^{\frac 2 3 \pi i} \sin(\pi /3) \frac{\Ga(\frac 1 3)}{\left(\frac 1 6 iy\right)^{1/3}}\right)\] as $y \rightarrow \infty$.  Simplifying gives the desired result for the case $\theta = \pi/2$. 

This gives the desired result in all cases.
 
\end{proof}

\subsection{Second case:  $0 < y < t$}\label{subsecSecondCaseYLessT}

Let define the constant $\mu>0$ by $t = y \cosh \mu$ and the function \[\psi(u) := \cosh u \cos w +w \cosh \mu.\]  We start by finding the saddle points and a suitable path.

Using (\ref{IntRepKBessel}), we have \begin{align}\label{IntRepKBessel3a}
 K_\nu(y) = \frac 1 2 \int_{-\infty}^\infty e^{-y \phi(R)} e^{rR} ~\wrt R  \end{align}where \[\phi(R):= \cosh R - i R\cosh \mu.\]
 
 The saddle points (values of $R$ for which $\phi'(R)=0$) are as follows~\cite{Tem} (see also~\cite[Section~2.1]{BST}
): \[R^\pm_k := \pm \mu + i\left( \frac \pi 2 + 2k\pi\right), \quad k \in \ZZ.\]

Let us now write $R=u+iw$ and thus we have \begin{align*}
\Re(-\phi(R)) =& -\cosh u \cos w  -w \cosh \mu = - \psi(u),\\
  \Im(-\phi(R)) =& -\sinh u \sin w + u \cosh \mu. \end{align*}

The paths of steepest descent/ascent through the saddle points $R^\pm_k$ is given by $ \Im(-\phi(R)) =  \Im(-\phi(R^\pm_k))$ and is the following family of curves~\cite{Tem}: \[ \sin w = \cosh \mu \frac {u}{\sinh u} \pm \frac{\sinh \mu - \mu \cosh \mu}{\sinh u}.\]  We use only the parts of these curves as shown in~\cite[Figure~3.3]{Tem}, which we will refer to as the path of steepest descent.  Notice that this path is the union of two branches $\eL^- \cup \eL^+$, separated by the imaginary axis, where \begin{align*}
 \textrm{ --- }& \eL^- \textrm { runs from } -\infty \textrm { to } 0 \textrm { and from } 0 \textrm{ to } + i\infty, \\  \textrm{ --- }& \eL^+ \textrm { runs from } +i\infty \textrm { to } 0 \textrm { and from } 0 \textrm{ to } + \infty. \end{align*}  What is important about this path is that, on both of the branches, the function $y\phi(R)$ has constant imaginary part, namely  \begin{align*} \chi := \Im(y\phi(R_0^+)):=&y \left(\sinh \mu  - \mu \cosh \mu \right) \\ =& y \sinh \mu - t \cosh^{-1}\left(\frac t y\right)= \sqrt{t^2 - y^2}  - t \cosh^{-1}\left(\frac t y\right), \\   \chi_- := \Im(y\phi(R_0^-)) =& - \chi
  \end{align*} for $\eL^+$ and $\eL^-$, respectively.

\begin{proof}[Proof of Theorem~\ref{thmSecondCaseOscillKBessel}]   We will use Laplace's method, which can be found at~\cite[Page~127, Theorem~7.1]{Olv}.

Using (\ref{IntRepKBessel3a}), we note that the integral representation is the correct form to apply Laplace's method.  We will use what we called the path of steepest descent as the path of integration; see~\cite[Figure~3.3]{Tem} for the graph.  We split this path of integration into two parts, the first from $-\infty$ to $i \infty$ and the second from $i \infty$ to $\infty$.  We apply Laplace's method separately to the two integration paths and, by the Cauchy-Goursat theorem, add the results together.   Let us first consider the integration path from $i\infty$ to $\infty$.  We see that conditions (i) -- (iv) of Laplace's method are satisfied.  

It remains to show that the final condition (v) is also satisfied.  We know that the integration path is a path comprised of steepest descent/ascent pieces because along it $\Im(\phi(u+iw))$ is constant and, when $u \rightarrow \infty$, we have that $\Re(\phi(u+iw)) \rightarrow \infty$.  As $R^+_0$ is a saddle point lying on the path, then $R^+_0$ is a local minimum on the path and, moreover, other local minima occur at the other saddle points (see~\cite[Page~66]{Cop}), which for us are $R^+_k$ where $k \in \NN$.  Directly computing, we see that $\Re(\phi(R^+_k) > \Re(\phi(R^+_0)$ for all $k \in \NN$.  Hence, $R^+_0$ gives the global minimum.  Thus, condition (v) is satisfied and we may apply Laplace's method to obtain \[e^{-y \phi(R^+_0)} \Ga\left(\frac 1 2 \right) \frac{a_0}{\sqrt{y}}\] as $y \rightarrow \infty$ where $\phi (R) = \cosh R-i R \cosh \mu$ and \[a_0= \frac{e^{rR}}{(2 \phi'')^{1/2}}\] evaluated at $R^+_0$.  Computing, we have that the contribution from this part of the path of integration to the dominant term is the following: \[\frac{\sqrt{\pi} e^{r \mu +i r\frac\pi 2}}{\sqrt{2 i y \sinh \mu}} e^{-y\left(\frac \pi 2 \cosh \mu +i (\sinh \mu - \mu \cosh \mu) \right)}.\] 

Likewise, for the other part of the path of integration, Laplace's method gives \[\frac{\sqrt{\pi} e^{-r \mu +i r\frac\pi 2}}{\sqrt{-2 i y \sinh \mu}} e^{-y\left(\frac \pi 2 \cosh \mu -i (\sinh \mu - \mu \cosh \mu) \right)}.\]  Here the saddle point which gives the global minimum is $R^-_0$ and the other saddle points $R^-_k$ where $k \in \NN$ are larger and can be ignored as before.

Adding these two parts together yields the desired dominant term of the asymptotic expansion for $K_\nu(y)$.  This gives the desired result.

\end{proof}

\section{Uniform bounds for  $K_\nu(y)$ where $\Im(\nu)$ large, $\Re(\nu)$ bounded, and $y$ is real and positive near coalescing saddle points}\label{subsecBndsBesselUnif} 



Already, when $r=0$, C.~Balogh computed a uniform asymptotic expansion which is valid  for all cases including the case of two nearby saddle points~\cite{Bal66}.  Balogh used a technique involving differential equations, but it is not clear that such a technique will work when $r$ is no longer zero.  We will use another technique, developed by C.~Chester, B.~Friedman, and F.~Ursell~\cite{CFU57}, which will yield the uniform dominant and next dominant terms for the case where $t$ and $y$ are nearly equal (or equal) and $r$ bounded.

\subsection{First case:  $y \geq t \geq 0$}  We prove Theorem~\ref{thmFirstCaseMonoKBesselUnif} in this section. Let \[F(R):=F(R, \theta):= -\cosh R + iR \sin \theta.\]  Then we have that  \begin{align}
 K_\nu(y) = \frac 1 2 \int_{-\infty}^\infty e^{y F(R)} e^{rR} ~\wrt R.\end{align}  The path of steepest descent has been obtained by N.~M.~Temme~\cite[Figure~2.1]{Tem} and the saddle points (values of $R$ for which $F'(R)=0$) that are relevant are $R_0:=i \theta$ and $R_1:=i(\pi - \theta)$.  Note that $R_0$ and $R_1$ are close in the complex plane when $\theta$ is close to $\pi/2$.  The technique used to estimate the $K$-Bessel function in Theorems~\ref{thmFirstCaseMonoKBessel} and~\ref{thmSecondCaseOscillKBessel} depends on the distance between $R_0$ and $R_1$ and, hence, does not yield a uniform estimate.


To use the Chester-Friedman-Ursell technique, let us introduce \begin{align*}\widetilde{\theta} &= \theta - \pi/2 \\ S & = 2^{-1/3}\left(iR +  \pi/2\right)
  \end{align*} where $2^{1/3}(1 - \cos \widetilde{\theta})$ assumes the role of the parameter $\alpha$ from the Chester-Friedman-Ursell technique (see~\cite[(3.2)]{CFU57}).  Under the change of variable from $R$ to $S$, the integral becomes \begin{align}\label{IntRepKBessel2CFU}
 K_\nu(y) =\int_{-i\infty+2^{-4/3}\pi}^{i\infty+2^{-4/3}\pi} \frac {-i} {2^{2/3}} e^{ir\left(\pi/ 2 - 2^{1/3}S\right)} e^{y F(i\pi/ 2 - i2^{1/3}S, \widetilde{\theta}+\pi/ 2)}~\wrt S.\end{align} and the relevant saddle points become \begin{align} \label{eqnSaddlePointsCFU}
S_0 &= 2^{-1/3}\left(iR_0 +  \pi/2\right) = -2^{-1/3}\widetilde{\theta} \\ \nonumber S_1 &=  2^{-1/3}\left(iR_1 +  \pi/2\right)  = 2^{-1/3}\widetilde{\theta}.  \end{align}  We will represent $F(i\pi/ 2 - i2^{1/3}S, \widetilde{\theta}+\pi/ 2)$ by the cubic~\cite[(2.1)]{CFU57} \begin{align}\label{eqnCFURep} F(i\pi/ 2 - i2^{1/3}S, \widetilde{\theta}+\pi/ 2) = \frac 1 3 u^3 - \zeta(\widetilde{\theta})u+ A(\widetilde{\theta})
  \end{align} where, under this representation, the saddle points correspond as follows:  \begin{align*} S_0 &\leftrightarrow u=\zeta^{\frac 1 2} (\widetilde{\theta}) \\ S_1 &\leftrightarrow u=-\zeta^{\frac 1 2} (\widetilde{\theta}).  
  \end{align*} By substitution in (\ref{eqnCFURep}), we have \begin{align*} F(i\pi/ 2 - i2^{1/3}S_0, \widetilde{\theta}+\pi/ 2) &= -\frac 2 3 \zeta^{\frac 3 2}(\widetilde{\theta})+ A(\widetilde{\theta}) \\ F(i\pi/ 2 - i2^{1/3}S_1, \widetilde{\theta}+\pi/ 2) &= \frac 2 3 \zeta^{\frac 3 2}(\widetilde{\theta})+ A(\widetilde{\theta}),
  \end{align*} which yields \begin{align*} A(\widetilde{\theta}) &= - \frac \pi  2 \cos \widetilde{\theta} \\  \zeta(\widetilde{\theta}) &= \left(\frac 3 2 \left( \widetilde{\theta} \cos \widetilde{\theta} -\sin \widetilde{\theta} \right)\right)^{2/3}.
  \end{align*}  Here, we have taken the branch of $\zeta(\widetilde{\theta})$ required by the Chester-Friedman-Ursell technique (see the top of page~603 in~\cite{CFU57}).  (And thus $\zeta(\widetilde{\theta})$ are positive real numbers.) Note that \[ \zeta(\widetilde{\theta}) \sim \frac{\widetilde{\theta}^2}{2^{2/3}} \sim 2^{1/3}(1 - \cos \widetilde{\theta}) \text{ as } \widetilde{\theta} \rightarrow 0.\]  Locally, the representation is analytic in $u$ which yields~\cite[(2.2)]{CFU57} \begin{align}\label{eqnCFUMain}\frac {-i} {2^{2/3}} e^{ir\left(\pi/ 2 - 2^{1/3}S\right)} \frac{\wrt S}{\wrt u} = \sum_{m=0}^\infty p_m(\widetilde{\theta})(u^2 - \zeta)^m +  \sum_{m=0}^\infty q_m(\widetilde{\theta})u (u^2 - \zeta)^m.
  \end{align}  Note that $-\widetilde{\theta}\geq 0$.  For small enough $-\widetilde{\theta}$ (independent of $y$ and $t$)~\cite[Lemma]{CFU57}, we have that the dominant term of the asymptotic expansion of $K_\nu(y)$ is~\cite[(5.2 -- 5.4), Theorem~2]{CFU57} \[2\pi ie^{yA(\widetilde{\theta})} p_0(\widetilde{\theta}) \frac{\Ai(y^{2/3}\zeta)}{y^{1/3}}\] and the next term is \[-2\pi i e^{yA(\widetilde{\theta})} q_0(\widetilde{\theta})\frac{\Ai'(y^{2/3}\zeta)}{y^{2/3}}.\]  
  
We now compute $p_0(\widetilde{\theta})$ and $q_0(\widetilde{\theta})$.  Taking first and second derivatives in (\ref{eqnCFURep}), we have\begin{align*} 2^{1/3}\left( \cos \widetilde{\theta} - \cos (2^{1/3}S)  \right) \frac{\wrt S}{\wrt u} = u^2 - \zeta \\ 2^{2/3}
\sin (2^{1/3}S) \left(\frac{\wrt S}{\wrt u}\right)^2 +  2^{1/3}\left( \cos \widetilde{\theta} - \cos (2^{1/3}S)  \right)  \frac{\wrt{}^2 S}{\wrt u^2} = 2 u.
  \end{align*} Substituting the two saddle points into the second derivative equation yields \begin{align*} \left(\frac{\wrt S}{\wrt u} \bigg\rvert_{u= \zeta^{1/2}}\right)^2 =\frac {2^{1/3} \zeta^{1/2}}{\sin(-\widetilde{\theta})}= \left(\frac{\wrt S}{\wrt u} \bigg\rvert_{u= -\zeta^{1/2}}\right)^2.
    \end{align*}  We now wish to determine the signs of square roots of these two expressions.  The Chester-Friedman-Ursell technique gives us that our representation is locally uniformly analytic in $S$ and $\widetilde{\theta}$, and thus we may take the limits $S \rightarrow 0$ and $\widetilde{\theta} \rightarrow 0$ in either order in the first derivative equation.  Now we have that $S =0 \leftrightarrow u=0$ (see the top of page~605 in~\cite{CFU57}).  Taking first $S \rightarrow 0$, we conclude that \begin{align*} \frac{\wrt S}{\wrt u} \bigg\rvert_{u= \zeta^{1/2}} =  \sqrt{\frac {2^{1/3} \zeta^{1/2}}{\sin(-\widetilde{\theta})}} = \frac{\wrt S}{\wrt u} \bigg\rvert_{u= -\zeta^{1/2}}
    \end{align*} for all  $-\widetilde{\theta}$ small enough.  
    
 Now plugging in the two saddle points into (\ref{eqnCFUMain}), we solve for $p_0(\widetilde{\theta})$ and $q_0(\widetilde{\theta})$:  \begin{align*} p_0(\widetilde{\theta}) &= \frac{\frac {-i} {2^{2/3}}\left(e^{ir\left(\pi/ 2 - 2^{1/3}S_0\right)} +e^{ir\left(\pi/ 2 - 2^{1/3}S_1\right)} \right)\sqrt{\frac {2^{1/3} \zeta^{1/2}}{\sin(-\widetilde{\theta})}} } 2 = \frac{-i}{\sqrt 2} e^{i r \pi/2} \cos (r \widetilde{\theta}) \sqrt{\frac {\zeta^{1/2}}{\sin(-\widetilde{\theta})}} 
 \\ q_0(\widetilde{\theta}) &= \frac{\frac {-i} {2^{2/3}}\left(e^{ir\left(\pi/ 2 - 2^{1/3}S_0\right)} -e^{ir\left(\pi/ 2 - 2^{1/3}S_1\right)} \right)\sqrt{\frac {2^{1/3} \zeta^{1/2}}{\sin(-\widetilde{\theta})}} } {2 \zeta^{1/2}} = \frac{1}{\sqrt 2} e^{i r \pi/2} \sin (r \widetilde{\theta}) \zeta^{-1/2}\sqrt{\frac {\zeta^{1/2}}{\sin(-\widetilde{\theta})}}.
  \end{align*}
  
  Thus, the desired dominant term and the next dominant term, respectively, are \begin{align}\label{eqnDomNextDomTermsCFU} &\frac {\pi \sqrt{2}} {y^{1/3}} e^{-y \frac \pi 2 \cos \widetilde{\theta} +i r \frac \pi 2}  \cos (r \widetilde{\theta}) \sqrt{\frac {\zeta^{1/2}}{\sin(-\widetilde{\theta})}} \Ai(y^{2/3} \zeta)
  \\ \nonumber &\frac {-i\pi \sqrt{2}} {y^{2/3}} e^{-y \frac \pi 2 \cos \widetilde{\theta} +i r \frac \pi 2}  \sin (r \widetilde{\theta}) \zeta^{-1/2}\sqrt{\frac {\zeta^{1/2}}{\sin(-\widetilde{\theta})}} \Ai'(y^{2/3} \zeta).
  \end{align}
  
  Finally, to finish the first case, we need to show that outside of a small enough neighborhood of the two saddle points, the integral is negligible (see~\cite[Section~5]{CFU57}).  It is a routine calculation to see that, outside of the small enough neighborhood of the two saddle points, the integral is on the order of $e^{-t \widetilde{\alpha}}$ for some $\widetilde{\alpha} > \pi/2$.  This concludes the proof of the first case, namely Theorem~\ref{thmFirstCaseMonoKBesselUnif}.

  %


 \subsection{Second case:  $0 < y < t$} We prove Theorem~\ref{thmSecondCaseOscillKBesselUnif} in this section.  Let \[G(R):= G(R, \mu):= -\cosh R + i R\cosh \mu.\]
Using (\ref{IntRepKBessel}), we have \begin{align}\label{IntRepKBessel3}
 K_\nu(y) = \frac 1 2 \int_{-\infty}^\infty e^{y G(R)} e^{rR} ~\wrt R.\end{align}  The paths of steepest descent/ascent has been obtained by N.~M.~Temme~\cite{Tem}, and the saddle points that are relevant are $R_0 :=  \mu + i\pi /2$ and $R_1 :=  -\mu + i\pi /2$.
 
 
 It is a routine calculation to see that, outside of a small enough neighborhood of the two saddle points, the integral is on the order of $e^{-t \widetilde{\alpha}}$ for some $\widetilde{\alpha} > \pi/2$ and, thus, negligible.  To finish, we compute the dominant term and the next dominant term using the Chester-Friedman-Ursell technique.  Changing variables \begin{align*}\widetilde{\theta} &= -i \mu \\ S &= 2^{-1/3}(i R + \pi/2), \end{align*} in (\ref{IntRepKBessel3}), we obtain (\ref{IntRepKBessel2CFU}) and the relevant saddle points become (\ref{eqnSaddlePointsCFU}).  Now we have essentially transformed the second case into the first case. There are a few minor differences, which we now state.  When $\mu>0$, the parameter $2^{1/3}(1 - \cos \widetilde{\theta})$ is a negative real number, and thus the Chester-Friedman-Ursell technique requires us to take the branch of $\zeta(\widetilde{\theta})$ for which it is a negative real number.  Thus, $\frac {\zeta^{1/2}}{\sin(-\widetilde{\theta})}$ is a positive real number and the sign of $\sqrt{\frac {2^{1/3} \zeta^{1/2}}{\sin(-\widetilde{\theta})}}$ is determined in a similar way.  Thus, we obtain the dominant and next dominant terms in (\ref{eqnDomNextDomTermsCFU}).  This concludes the proof of the second case, namely Theorem~\ref{thmSecondCaseOscillKBesselUnif}.



\section{Bounds for  $K_{r-1/2+it}(y)$ for $0<y<1$ and $1/2 \leq r \leq 3/2$}\label{secBndKBesselSmally}

Finally, we give an estimate of $K_{r-1/2+it}(y)$ for small positive real argument.  Recall that we pick $t_0\geq1$ to be a fixed large constant (large enough to use the first term in the Stirling asymptotic series for the gamma function for the approximation below).

\begin{proof}[Proof of Proposition~\ref{lemmUBndsOnK}]
Since we need only need a bound for small $y$, it suffices to adapt the bound for purely imaginary order from~\cite[Section~3.1]{BST}.  It is well-known (see~\cite[Page~78, (6), Page~77, (2)]{Wa} for example) that the $K$-Bessel function is defined as \begin{align}\label{eqnKInTermsOfI}
K_{\nu}(z) := \frac 1 2 \pi \frac{I_{-\nu}(z) - I_\nu(z)}{\sin(\nu \pi)}  \end{align} where $I_\nu(z)$ is the modified Bessel function of the first kind \[I_\nu(z) := \sum_{m=0}^\infty \frac{(\frac1 2 z)^{\nu+2 m}}{m! \Ga(\nu + m+1)}.\]  (The $\Ga$ here is the gamma function, not the lattice subgroup.)

We would like to bound $K_{r - 1/2 +i t}(y)$.  An elementary identity gives a lower bound for \begin{align}\label{eqnSinSinchEst}
 \left|\sin\left(r\pi - 1/2\pi +i t\pi \right)\right|  = \left|\frac {-1} {2i} \left( e^{t \pi} e^{-i(r-1/2)\pi}- e^{-t \pi}  e^{i(r-1/2)\pi} \right)\right| \geq \frac 1 2 e^{|t| \pi} - \frac 1 2
\end{align} for all $t$.  

Taking the first term of the Stirling asymptotic series for the gamma function, we have \[\Ga(s) = \sqrt{2 \pi} s^{s - 1/2} e^{-s} e^{R(s)}\] where $R(s) = o(|s|^{-1})$.  Hence, we have \begin{align*}
|m!\Ga(r-1/2+ & it + m+1)| \\ &= 2 \pi m^{m+1/2} e^{-m} |r+1/2+m+it|^{r+m} e^{-t \arg(r+1/2+m+it)}e^{-r-1/2-m}e^{o(|r+1/2+m+it|^{-1})}  \\&  \geq C |t|^r e^{-\frac{|t|\pi}2}\end{align*} where the constant $0<C$ depends only on $t_0$.  Note that, since $r+1/2+m >0$, we have that $0\leq \arg(r+1/2+m+it)\leq \pi/2$ for $t>0$ and $ -\pi/2 \leq \arg(r+1/2+m+it)\leq0$ for $t<0$.  Likewise, we have \[|m!\Ga(-r+1/2-  it + m+1)| \geq C |t|^{1-r} e^{-\frac{|t|\pi}2}.\]

Now, for $0<y<2$, we have that $\sum_{m=0}^\infty (y/2)^{2m} \leq 4/(4 - y^2)$.  All of this now implies that \[|K_{r - 1/2 +i t}(y)| \leq \pi \frac{\frac 4 {4-y^2}\left((\frac y 2)^{r-1/2} + (\frac y 2)^{1/2-r}\right)\frac{e^{|t|\pi/2}}{C \min(|t|^r, |t|^{1-r})}}{e^{|t|\pi} -1} \leq \widetilde{C} y^{1/2-r} e^{-|t|\pi/2} |t|^{r-1},\]  where $\widetilde{C}$ depends only on $t_0$ and is uniformly bounded for all large enough $t_0$.  This is the desired result.

\end{proof}


\section{Eisenstein series}\label{secBndEisenstein}


\subsection{Bounds on the Fourier coefficients of Eisenstein series}


Using our result on the asymptotics of the $K$-Bessel function, we now give the proof of Theorem~\ref{lemmBndSumcnsquares}, namely a bound for the sum of the $c_n$.

\begin{proof}[Proof of Theorem~\ref{lemmBndSumcnsquares}]
The proof is an adaption of the proof of~\cite[Proposition~4.1]{St}, which is, itself, an adaption of~\cite[Proposition~5.1]{Wol}.  Let $0<Y<H$ be given and define \[J:= \int_\D \left|E^{(j)}_0(z,s)\right|^2 \frac{\wrt x~\wrt y}{y^2} \quad \textrm{ where } \D:= (0,1) \times (Y,H).\]  Let $B := \max(B_0, H, Y^{-1})$.  Since $E^{(j)}_0(z,s)$ is automorphic, we can apply exactly the same proof as in~\cite[Proposition~4.1]{St} to obtain \[J \leq O(1 + Y^{-1}) \int_{F_B} \left|E^{(j)}_0(z,s)\right|^2 \frac{\wrt x~\wrt y}{y^2},\] where, recall, $F_B$ is the bounded part (i.e. with cusps removed) of $\overline{F}$.  Let us define the modified Eisenstein series in which we remove the zeroth term of the Fourier expansion:  \[E^{(j)}_{0,B}(z,s):= \begin{cases} 
      E^{(j)}_0(z,s) & \textrm{if } z \in F_B\\
      E^{(j)}_0(z,s) - \delta_{jk} \left(Im (\sigma_k z)\right)^s - \varphi_{jk}(s) \left(Im (\sigma_k z)\right)^{1-s}  & \textrm{if } z \in \Ce_{k,B}.
   \end{cases}\]
   
 By the Maass-Selberg relation~\cite[Page~301~(3.43), Page~281]{He}, we have\begin{align*}
 \sum_{j=1}^q \int_{F_B} & \left|E^{(j)}_{0,B}(z,s)\right|^2 \frac{\wrt x~\wrt y}{y^2}  \\ & = \frac {1}{2r-1}\left(q B^{2r-1} - B^{1 -2r} \sum_{j=1}^q \sum_{j'=1}^q |\varphi_{j j'}(s)|^2\right) +  \sum_{j=1}^q Re\left( \overline{\varphi_{j j}(s)}\frac{B^{2it}}{it}\right). \end{align*}  Applying~\cite[Page~300 (3.38)]{He} yields \[ J \leq O(1 + Y^{-1}) \left(B^{2r-1} +\omega(t)\right).\]
 
 
Substituting the Fourier expansion of the Eisenstein series (\ref{eqnFourExpInYEisen2}) in the definition of $J$ and applying Parseval's formula yields \begin{align*} J \geq \sum_{n\neq0} |c_n|^2 \int_{2 \pi |n| Y}^{2 \pi |n| H} |K_{r-1/2+it}(y)|^2 \frac {\wrt y}{y}.\end{align*}  Let $Y = |t|/(8 \pi N)$ and $H = |t|/(4 \pi)$.  With this choice, we have \[\left[ |t|/4, |t|/2\right] \subset \left[2 \pi |n|Y, 2 \pi |n| H \right] \quad \textrm{ whenever } 1 \leq |n| \leq N,\] and, hence,\[\sum_{1 \leq |n|\leq N}|c_n|^2 \leq C^{-1}J \textrm{ where }C =  \int_{|t|/4}^{|t|/2} |K_{r-1/2+it}(y)|^2 \frac {\wrt y}{y}.\]  Theorem~\ref{thmSecondCaseOscillKBessel} now gives that \[C^{-1} \leq O(|t| e^{|t| \pi}).\]  Combining, we obtain the desired result:  \[\sum_{1 \leq |n|\leq N}|c_n|^2 = O\left(e^{|t|\pi} (N+|t|)\right)\left\{\omega(|t|) + \left(|t| + \frac{N}{|t|}\right)^{2r-1}  \right\}.\]

\end{proof}

\subsection{Bounds on Eisenstein series}


We now give the proof of Theorem~\ref{propEisenBiggerHalf}, namely a bound for the Eisenstein series themselves.  Recall that we defined $s:=r+it$.  



\begin{lemm}\label{lemmCoarseBndKBessel}
Fix $M>0$ and $y_0 >0$. Let $|r|\leq M$ and $y \geq y_0$.  Then \[K_{r+it}(y) = O\left(\frac{e^{-y}}{\sqrt y}\right)\] where the constant depends only on $M$ and $y_0$.
\end{lemm}
\begin{proof}
When $|R| > (24 |r| y^{-1})^{1/3}$, we have that \[h(R):=-\frac{y R^4}{24} + rR < 0,\] and, thus, on the complement, the function $h(R)$ is bounded by a constant $N(M, y_0)>0$.  Using the integral representation (\ref{IntRepKBessel}) for the $K$-Bessel function, we have \[|K_{r+it}(y)| \leq \frac 1 2 \int_{-\infty}^\infty e^{-y\left(1 + \frac{R^2} 2 + \frac{R^4}{24}\right)+rR}~\wrt R \leq \frac {e^N} 2 \int_{-\infty}^\infty e^{-y\left(1 + \frac{R^2} 2 \right)}~\wrt R.\]  The desired result now follows. \end{proof}

The following lemma gives some bounds for the $K$-Bessel function that are convenient for our proof of Theorem~\ref{propEisenBiggerHalf}.  

\begin{lemm}\label{lemmUpperBndBessel}
Let $|t| \geq t_0$.  We have \[K_{r-1/2+it}(y) = \begin{cases} O\left(y^{1/2-r} e^{-|t| \frac \pi 2} |t|^{r-1}\right) & \textrm{ if }  0 <y < 1  \textrm{ and } 3/2 \geq r \geq1/2 \\ O\left(e^{-|t| \frac \pi 2 }|t|^{r-5/6}\right) & \textrm{ if }  1 \leq y < \frac \pi 2 |t| \textrm{ and } 2/3 > r \geq 1/2 \\ O\left(e^{-|t| \frac \pi 2 }|t|^{r-1}\right) & \textrm{ if }  1 \leq y < \frac \pi 2 |t| \textrm{ and } 3/2 \geq r \geq2/3\\ O\left( \frac{e^{-y}}{\sqrt y}\right) &\textrm{ if } y \geq \frac \pi 2 |t|  \textrm{ and } 3/2 \geq r \geq1/2\end{cases}\]
\end{lemm} where the implied constants depend on $t_0$ in first, second, and third branches and has no dependance in the fourth branch.

\begin{proof}
For $0<y<1$, apply Proposition~\ref{lemmUBndsOnK}, and, for $y \geq \frac \pi 2 |t|$, apply Lemma~\ref{lemmCoarseBndKBessel}.

Let us now consider $|t| \leq y < \frac \pi 2 |t|$.  Let $\theta_0$ be as in Theorem~\ref{thmFirstCaseMonoKBesselUnif}.  For the range $\frac{|t|}{\sin(\pi/2 - \theta_0)} \leq y < \frac \pi 2 |t|$, we apply Theorem~\ref{thmFirstCaseMonoKBessel} with the observation that $-\sqrt{x^2 -1} + \arccos(1/x) < 0$ for $\frac \pi 2 > x>1$ to obtain \[K_{r-1/2+it}(y) = O\left(\frac{e^{-|t| \frac \pi 2}}{(y^2 - |t|^2)^{1/4}}\right) =  O\left(\frac{e^{-|t| \frac \pi 2}}{|t|^{1/2}}\right)\] where the implied constant depends only on $t_0$.  Note that the conclusion of Theorem~\ref{thmFirstCaseMonoKBessel} is uniform over all $0 < \theta \leq \pi/2 - \theta_1$ for any $\pi/2 >\theta_1>0$.  For the range $|t| \leq  y \leq \frac{|t|}{\sin(\pi/2 - \theta_0)}$, we apply Theorem~\ref{thmFirstCaseMonoKBesselUnif} to obtain \begin{align}\label{eqnUnifKBesselEstFirstCase}
 K_{r-1/2+it}(y) =  O\left(\frac{e^{-|t| \frac \pi 2}}{|t|^{1/3}}\right) \end{align} where the implied constant depends only on $t_0$.

Finally, let us consider $1 \leq y \leq |t|$.  Let $\mu_0$ be as in Theorem~\ref{thmSecondCaseOscillKBesselUnif}.  For the range $\frac{|t|}{\cosh \mu_0} \leq y \leq |t|$, we apply Theorem~\ref{thmSecondCaseOscillKBesselUnif} to also obtain (\ref{eqnUnifKBesselEstFirstCase}) where, likewise, the implied constant depends only on $t_0$.  For the range $1 \leq y \leq \frac{|t|}{\cosh \mu_0}$, we will apply Theorem~\ref{thmSecondCaseOscillKBessel} with two observations.  The first is that the conclusion

\begin{align*}
K_{r+it}(y) \sim \sqrt{\frac{2\pi}{y \sinh \mu}}e^{-y \frac \pi 2 \cosh \mu +i r \frac \pi 2}& \left[\cosh(r \mu) \sin\left(\frac \pi 4 - y \left(\sinh \mu  - \mu \cosh \mu \right)\right) \right. \\ &\left.\quad  -i \sinh(r \mu)\cos\left(\frac \pi 4 - y \left(\sinh \mu  - \mu \cosh \mu \right)\right)\right]
\end{align*} is also valid as $t \rightarrow \infty$.

 The second observation is that the conclusion of Theorem~\ref{thmSecondCaseOscillKBessel} is uniform over all positive $\mu$ bounded away from $0$, and, hence, is valid for $y>0$ arbitrarily close to $0$.   Applying Theorem~\ref{thmSecondCaseOscillKBessel}  for the range $1 \leq y \leq \frac{|t|}{\cosh \mu_0}$ yields \[K_{r-1/2+it}(y) = O\left(\frac{e^{-|t| \frac \pi 2}|t|^{r-1/2}}{\sqrt{|t|}}\right) =  O\left(e^{-|t| \frac \pi 2}|t|^{r-1}\right)\] where the implied constant depends only on $t_0$.   This gives the desired result.


%
%

\end{proof}


We also have the following bound, which we will use in the proof of Theorem~\ref{propEisenBiggerHalf}.

\begin{lemm}\label{lemmScatMatUnifBnd} For $3/2 \geq r \geq 1/2$, we have $\varphi_{jk}(r +it)$ is uniformly bounded for $|t| \geq 1$.
 
\end{lemm}
\begin{proof}
Apply~\cite[Page 301, (a)]{He}.
\end{proof}

Following the proof scheme of~\cite[Proposition~4.2]{St}, we can now bound the Eisenstein series:
\begin{proof}[Proof of Theorem~\ref{propEisenBiggerHalf}]
We now give the proof for $\frac 3 2 \geq r > \frac 1 2$, leaving the proof for $r= \frac 1 2$ to the end.  Consider three cases:  $0<y<1$,  $1 \leq y \leq \frac{|t|}2$, and $\frac{|t|}2 < y$.

The first case is $0<y<1$.  Let us consider the range $\frac 3 2 \geq r > 1$ first.  By Lemmas~\ref{lemmScatMatUnifBnd} and~\ref{lemmUpperBndBessel}, we obtain the following upper bound for (\ref{eqnFourExpInYEisen2}):  \begin{align}\label{eqnUpBndEisenstein} O(y^{1-r}) + O\left(y^{1-r}e^{-|t|\frac \pi 2}|t|^{r-1} \right) \sum_{n=1}^\infty \left(|c_n| + |c_{-n}| \right) f(n)\end{align} where \[f(X) := \begin{cases} 1 &\textrm{ if } X < \frac {|t|}{4y} \\ e^{|t| \frac \pi 2 - 2 \pi X y} &\textrm{ if } X \geq \frac {|t|}{4y} \end{cases}.\]  Now define \[S(X) := \sum_{1 \leq |n| \leq X} |c_n|.\]

By the fact that $f(X)$ is continuous and monotonically decreasing, that \[\sum_{n=1}^\infty \left(|c_n| + |c_{-n}| \right)\] can be written as a telescoping sum, that $S(X)$ is a function of bounded variation on any closed interval, that $S(1/2)=0$, and that $f(X) S(X) \rightarrow 0$ as $X \rightarrow \infty$ (which follows from Theorem~\ref{lemmBndSumcnsquares} and the Cauchy-Schwarz inequality), we can apply the definition of the Riemann-Stieltjes integral to obtain the inequality and integration by parts to obtain the equality:  \begin{align}\label{eqnBndForCn}
\sum_{n=1}^\infty \left(|c_n| + |c_{-n}| \right) f(n) \leq \int_{1/2}^\infty f(X) \wrt S(X) = -\int_{1/2}^\infty f'(X) S(X) \wrt X.
\end{align}

To bound (\ref{eqnBndForCn}), it suffices to estimate $S(X)$ for $X \geq \frac {|t|}{4y}$ using Theorem~\ref{lemmBndSumcnsquares} and the Cauchy-Schwarz inequality:  \begin{align*}
 S(X) = O\left(e^{|t| \frac \pi 2}\right) \left(X^{r + 1/2}+X \sqrt{\omega(t)}\right).
  \end{align*}  Using calculus, we obtain \[\sum_{n=1}^\infty \left(|c_n| + |c_{-n}| \right) f(n) \leq O\left(e^{|t|\frac \pi 2}\right)\left(\left(\frac{|t|} y\right)^{r+1/2} +\frac{|t|} y \sqrt{\omega(t)}\right),\] which yields the desired result for the range $\frac 3 2 \geq r > 1$.
  
  For the desired result in the range $1 \geq r > \frac 1 2$, replace (\ref{eqnUpBndEisenstein}) with  \begin{align*} O(y^{1-r}) + O\left(y^{1-r}e^{-|t|\frac \pi 2} \right) \sum_{n=1}^\infty \left(|c_n| + |c_{-n}| \right) f(n)\end{align*} in the proof for the range $\frac 3 2 \geq r > 1$.  This proves the first case $0<y<1$.
  
The second case is $1 \leq y \leq \frac{|t|}2$.  Replace (\ref{eqnUpBndEisenstein}) with \begin{align*}\begin{cases} \delta_{j1}y^{r+it} + O(y^{1-r}) + O\left(\sqrt{y} e^{-|t|\frac \pi 2} \right) \sum_{n=1}^\infty \left(|c_n| + |c_{-n}| \right) f(n) & \textrm{ if } 1\geq r >  \frac 1 2\\ \delta_{j1}y^{r+it} + O(1) + O\left(\sqrt{y} e^{-|t|\frac \pi 2} |t|^{r-1}\right) \sum_{n=1}^\infty \left(|c_n| + |c_{-n}| \right) f(n) & \textrm{ if } \frac 3 2 \geq r > 1\end{cases}.\end{align*}  In the case $1 \leq y \leq \frac{|t|}2$, we have that  \begin{align*}
 S(X) = O\left(e^{|t| \frac \pi 2}\sqrt{y} \right) \left(X^{r + 1/2} y^{r-1/2}+X \sqrt{\omega(t)}\right),
  \end{align*} which yields  \[\sum_{n=1}^\infty \left(|c_n| + |c_{-n}| \right) f(n) \leq O\left(e^{|t|\frac \pi 2} y^{-1/2}\right)\left(|t|^{r+1/2} +|t| \sqrt{\omega(t)}\right)\] and the desired result for the second case $1 \leq y \leq \frac{|t|}2$.
  
  The third case is $\frac{|t|}2 < y$.  For $\frac 3 2 \geq r > \frac 1 2$, replace (\ref{eqnUpBndEisenstein}) with \begin{align*}\delta_{j1}y^{r+it} + O(y^{1-r}) + O\left(1\right) \sum_{n=1}^\infty \left(|c_n| + |c_{-n}| \right) f(n) \end{align*} and replace the previous $f(X)$ with \[f(X) = \frac{e^{-2\pi X y}}{\sqrt{2 \pi X}}.\]  In the case that $\frac{|t|}2 < y$, we have that \begin{align*}
 S(X) = O\left(e^{|t| \frac \pi 2} \right) \left(X^2 + \sqrt{|t|}X^{\frac 3 2} + \gamma(t) X + \gamma(t) \sqrt{|t|} X^{\frac1 2}\right),
  \end{align*} where $\gamma(t):= \sqrt{\omega(t)} +|t|$.  Then \begin{align*}\label{eqnBndForCn2}
\sum_{n=1}^\infty \left(|c_n| + |c_{-n}| \right) f(n) \leq \int_{1}^\infty f(X) \wrt S(X)
\end{align*} holds and the desired result for the third case  $\frac{|t|}2 < y$ now follows by calculus.  This completes the proof of the theorem for $\frac 3 2 \geq r > \frac 1 2$.

For $r=\frac 1 2$, the proof is analogous to that of $1 \geq r > \frac 1 2$, except we replace Theorem~\ref{lemmBndSumcnsquares} with~\cite[Proposition~4.1]{St}, yielding, for every $\varepsilon>0$, the following estimate for $S(X)$:  \[S(X) = O\left(e^{|t| \frac \pi 2} |t|^\varepsilon \sqrt{\omega(t)}X^{\frac 1 2 + \varepsilon} \sqrt{X + |t|}\right),\] which holds for every $X \geq \frac 1 2$.  With this change, the proofs of the three cases ($0<y<1$,  $1 \leq y \leq \frac{|t|}2$, and $\frac{|t|}2 < y$) are analogous. This completes the proof of the theorem.

\end{proof}

\end{document}